\documentclass[12pt,a4paper,oneside]{amsart}
\usepackage{amsfonts, amsmath, amssymb, amsthm}
\usepackage[foot]{amsaddr}
\usepackage[colorlinks=true,citecolor=blue]{hyperref}
\usepackage[margin=1.4in]{geometry}
\usepackage{txfonts}
\usepackage[T1]{fontenc}
\usepackage{bbm}
\usepackage{mathtools}

\usepackage{graphicx}
\usepackage{enumerate}
\usepackage{verbatim}
\usepackage{tikz}
\usepackage{algorithmic}

\usetikzlibrary{decorations,decorations.pathreplacing}

\begingroup
    \makeatletter
    \@for\theoremstyle:=definition,remark,plain\do{%
        \expandafter\g@addto@macro\csname th@\theoremstyle\endcsname{%
            \addtolength\thm@preskip\parskip
            }%
        }
\endgroup

\newtheorem{theorem}{Theorem}
\newtheorem*{theorem*}{Theorem}
\newtheorem{lemma}[theorem]{Lemma}
\newtheorem{prop}[theorem]{Proposition}

\newtheorem{corollary}[theorem]{Corollary}
\newtheorem*{question*}{Question}

\theoremstyle{definition}
\newtheorem{definition}[theorem]{Definition}
\newtheorem{remark}[theorem]{Remark}
\newtheorem*{remark*}{Remark}

\newtheorem{example}[theorem]{Example}
\newtheorem{problem}[theorem]{Problem}

\newcommand{\cal}{\mathcal}
\newcommand{\conv}[1]{\mathrm{conv}( {#1})}

\newcommand{\es}[1]{S_{\hspace{-.2ex}{#1}}}

\begin{document} 

\title{Helly type problems in convexity spaces}  

\author{Andreas F. Holmsen}
\date{\today}

\address{\noindent Department of Mathematical Sciences, KAIST,
Daejeon, South Korea \hfill \hfill  \linebreak \and \linebreak
Discrete Mathematics Group,  Institute for Basic Sciences (IBS), Daejeon, South Korea.}
%\email{andreash@kaist.edu}

\begin{abstract} 
We report on some  recent progress regarding combinatorial properties in convexity spaces with a bounded Radon number. In particular, we discuss the relationship between the Radon number, the colorful and fractional Helly properties, weak $\varepsilon$-nets, and $(p,q)$-theorems. As an application of the theory of convexity spaces we introduce some nnew classes of uniform hypergraphs and show that they are $\chi$-bounded.
\end{abstract}

\maketitle

\section{Introduction}

\subsection{Background}
The closely related theorems of Helly \cite{helly} and Radon \cite{radon} lie at the foundation of the combinatorial geometry of convex sets in $\mathbb{R}^d$. Since their discovery, more than a century ago, there now exists an extensive literature concerning their generalizations, variations, and applications. For further details, references, and historical notes we refer the reader to one of the surveys \cite{ADLS, bara-sob, lgmm2019,  eckhoff-surv, holmsen-wenger} as well as the textbooks \cite{barabook, mato}.

\smallskip

One of the landmark results of modern combinatorial convexity is the {\em $(p,q)$-theorem} of Alon and Kleitman \cite{pq-alon}, originally conjectured by Hadwiger and Debrunner \cite{had_deb}. Informally speaking, the $(p,q)$-theorem says that if the intersecting $(d+1)$-tuples of a finite family of convex sets in $\mathbb{R}^d$ are {\em evenly distributed}, quantified in terms of the integers $p$ and $q$, then the transversal number of the family is bounded by a function depending only on $p$, $q$, and $d$. This should be held in comparison to Helly's original theorem which asserts that if {\em every} $(d+1)$-tuple is intersecting, then the transversal number of the family equals one, that is, there exists a single point that intersects every member of the family. See \cite{eckhoff-pq} for a historical background on the $(p,q)$-theorem and related questions.

\smallskip

Alon and Kleitman's celebrated proof of the $(p,q)$-theorem has been scrutinized over the years, revealing that the proof does not actually rely much on the geometry of $\mathbb{R}^d$, but rather on certain combinatorial properties of families of convex sets. For instance, the work of Alon, Kalai, Matousek, and Meshulam \cite{akmm} highlights the importance of the combinatorial aspects of the {\em fractional Helly theorem} due to Katchalski and Liu \cite{katch-liu}. 

\smallskip

As it turns out, there already exists a framework for exploring the underlying combinatorial principles of the $(p,q)$-theorem, namely a type of set systems called {\em convexity spaces}. Indeed, the notion of a convexity space was introduced by Levi  \cite{levi} in the 1950's to show that Helly's theorem is a purely combinatorial consequence of Radon's theorem. In the 1970's there was a renewed interest in convexity spaces, in part motivated by the so-called {\em partition conjecture} (also commonly referred to as Eckhoff's conjecture) which would imply a purely combinatorial proof of Tverberg's famous theorem \cite{tverberg}. See \cite{eckhoff-partition} for a historical account of the partition conjecture as of the early 2000's.

\smallskip

The partition conjecture itself was famously disproved by Bukh \cite{bukh} in 2010. In the same manuscript Bukh improved the bounds on the Tverberg numbers for convexity spaces and raised several tantalizing open problems, which incidentally would be highly relevant to the understanding of the combinatorics  of convexity spaces and the $(p,q)$-theorem.

\subsection{Outline of the paper} There are two main parts to this paper. The first part (section 2) acts as a survey. Here we give a formal definition of a convexity space and introduce the main invariants and properties we will be dealing with. This will be accompanied by an account of the recent developments on the relations between these invariants and properties. 
The main takeaway here is Theorem \ref{t:equivalences} which is the accumulation of the works of several independent groups spanning over the last two decades. 

\smallskip

The second part (section 3) contains some new results about colorings of hypergraphs. 
More specifically, we introduce some classes of uniform hypergraphs and show that these are $\chi$-bounded, meaning that their chromatic number is bounded by a function of the clique number. 
To the best of our knowledge, these results have not appeared anywhere before. Their inclusion here is due to the fact that the proofs rely on the theory of convexity spaces, in particular the results of section 2.  

\smallskip

We wrap up with some concluding remarks and questions in section 4.

%\smallskip

\subsection{Notation and terminology} We use the following standard notation and terminology. For a positive integer $n$ we write $[n]$ to denote the set $\{1, 2, \dots, n\}$. For a finite set $S$ and an integer $k$ we let $\binom{S}{k}$ denote the family of {\em $k$-tuples} (i.e. $k$-element subsets) of $S$. 
Given a set system $\cal S$ on a ground set $X$, we write $\cup \cal S$ and $\cap \cal S$ to denote the union $\bigcup_{K\in \cal S} K$ and intersection $\bigcap_{K\in \cal S} K$, respectively. We say that $\cal S$ is {\em intersecting} if $\cap \cal S \neq \emptyset$.

\smallskip
The terms {\em set system} and {\em family} (of sets) will be used frequently. 
In general, when using the term set system we allow for {\em infinite} set systems. The term family, on the other hand, will be reserved for {\em finite} families of sets.  

%\smallskip

\subsection{Acknowledgements}
This work was partially supported by ERC Advanced Grants ``GeoScape'' 882971 and ``ERMiD'' 101054936, and by the Institute for Basic Science (IBS-R029-C1).
The author wishes to thank the organizers and participants of the {\em Erd{\H o}s Center Focused week: Combinatorial geometry in Radon convexity spaces}, which took place in Budapest, December 2023. 

A special thanks goes to Attila Jung who gave many valuable comments on the original draft of this manuscript. He suggested an alternative definition of the hypergraph property $T_k(m)$ (Definition \ref{d:tkm}) which 
greatly simplified the proof of Theorem \ref{t:chi-bounded}. He also suggested to bound the colorful Helly number directly (Proposition \ref{p:color_bound}), rather than bounding the Radon number (Proposition \ref{p:radon_bound}) as was done in the original draft. 

\section{Convexity spaces: invariants and properties}

\subsection{Basic definitions}
A {\em convexity space} is a set system  $\cal C$ on a (nonempty) ground set $X$ which satisfies the following properties:

\begin{enumerate}
\item[(C1)]\label{ax:empty} $\emptyset, X \in \cal C$.
\item[(C2)]\label{ax:intersection} If $\emptyset\neq \cal D \subset \cal C$, then $\cap \cal D \in \cal C$.
\item[(C3)]\label{ax:nested} If $\emptyset\neq \cal D\subset \cal C$ is totally ordered by inclusion, then $\cup \cal D \in \cal C$.
\end{enumerate}
We denote a convexity space by the pair $(X, \cal C)$ and the members of $\cal C$ are called {\em convex sets}. For an in-depth treatment of convexity spaces we refer the reader to van de Vel's monograph \cite{vandevel}.

\begin{example} \label{ex:spaces} Here we give a few examples of convexity spaces. Further examples can be found in e.g. \cite[section 1.5]{vandevel} and \cite[section 2]{eckhoff-partition}.

\begin{enumerate}
\item \label{ex:standard} The most obvious example of a convexity space is the {\em standard convexity}  which 
occupies the majority of discrete geometry and combinatorial convexity. Here $X = \mathbb{R}^d$ and $\cal C$ consists of the ``usual'' convex sets in $\mathbb{R}^d$. That is, $S\subset \mathbb{R}^d$ is convex provided for all $x, y\in S$ and $t\in [0,1]$ we have $(1-t)x+ty \in S$. 

\smallskip

\item \label{ex:lattice} The {\em lattice convexity} is obtained by setting $X = \mathbb{Z}^d$ and letting $\cal C$ consist of all sets of the form $S\cap \mathbb{Z}^d$ where $S$ is a standard convex set in $\mathbb{R}^d$. Helly numbers and related invariants of this convexity space have been studied in \cite{lattice, JPD, onn}.
A generalization of lattice convexity can be obtained by replacing $\mathbb{Z}^d$ by any subset of $\mathbb{R}^d$. See for instance \cite{ambrus, dillon, DeLorian} for some recent interesting results and research problems.

\smallskip

\item \label{ex:box} The {\em box convexity} is a subspace of the standard convexity where $\cal C$ consists of all axis-parallel boxes in $\mathbb{R}^d$, that is, Cartesian products of segments parallel to the coordinate axes. Radon numbers and related invariants of these spaces were investigated in   \cite{eckhoff-boxes, hare-thompson, jamison}.

\smallskip

\item \label{ex:closure} Let $F$ be a nonempty finite family of subsets of a ground set $X$, and let $F^\cap = \{\cap G : G\subset F\}  \cup \{\emptyset, X\}$.
Weak $\varepsilon$-nets and $(p,q)$-theorems in the finite convexity space $(X, F^\cap)$ were investigated in \cite{akmm}.

\smallskip

\item \label{ex:dual} Dual to the previous example, let $F$ be a nonempty finite family of subsets of a ground set $X$ and suppose that $F$ is an antichain with respect to containment. For a subset $Y \subset X$ let $F_Y = \{S \in F: Y \subset S\}$ be the subfamily of $F$ whose members contain $Y$. By setting $\cal C =\{ F_Y : Y\subset X \}$, we obtain a finite convexity space $(F, \cal C)$. This convexity space appears implicitly in Bukh's counter-example to the partition conjecture \cite{bukh}. 
In section 3 we examine this space in greater detail for the case when $F$ is the family of maximal independent sets of a uniform hypergraph.
\end{enumerate}
\end{example}

\medskip

\noindent 
Given a convexity space $(X, \cal C)$, it follows from property (C1) that any subset $Y\subset X$ is contained in at least one convex set. Furthermore, property (C2) implies that $Y$ is contained in a unique smallest convex set
\[\conv{Y} = \textstyle \cap\{K : K \in {\cal C}, Y\subset K\},\]
which we refer to as the {\em convex hull} of $Y$.

\begin{remark} \label{r:closure}
A set system that satisfies properties (C1) and (C2) is often referred to as a {\em closure system}. 
Note that these two properties alone suffice to define the convex hull operator,  
which is usually referred to as the closure operator in this context. Property (C3) is equivalent to requiring that the convex hull (closure) operator is domain finite, meaning that for any $Y\subset X$ and $p\in \conv Y$ there is a {\em finite} subset $Y'\subset Y$ such that $p\in \conv {Y'}$. See \cite[Theorem 1.3]{vandevel} for details.
\end{remark}

\subsection{Radon numbers} Here we discuss some classical invariants of convexity spaces motivated by Radon's theorem. As we will see in Theorem \ref{t:equivalences}, these invariants are related to several deep structural properties of a convexity space. 

\begin{definition}
Let $Y$ be a (multi)subset of $X$ in a convexity space $(X, \cal C)$. A {\em Radon partition} of $Y$ is a partition $Y = A \cup B$ such that 
\[\conv A \cap \conv B \neq \emptyset.\]
The {\em Radon number} $r(X, \cal C)$ is the minimal integer $n$ (if it exists) such that every subset $Y\subset X$ with $|Y|\geq n$ admits a Radon partition.
\end{definition}

If there exists arbitrarily large subsets with no Radon partition, we write $r(X, \cal C) = \infty$. 
Note that when $X$ is finite we always have $r(X,\cal C) \leq |X|+1$. 

\smallskip

One of the basic facts about the standard convexity on $\mathbb{R}^d$ is that its Radon number equals $d+2$, which is known as Radon's theorem (or Radon's lemma) \cite{radon}. 
The Radon number for the box convexity on $\mathbb{R}^d$ is known and grows asymptotically as $\Theta(\log d)$, 
see e.g. \cite[equation (3.1)]{eckhoff-partition} for the precise expression.
In contrast, the Radon number for the lattice convexity on $\mathbb{Z}^d$ is not known, but it is bounded below by $\Omega(2^d)$ and above by $O(d2^d)$ \cite{onn}.

\begin{definition}
Let $Y$ be a (multi)subset of $X$ in a convexity space $(X, \cal C)$.
For $k\geq 2$, a {\em Tverberg $k$-partition} of $Y$ is a partition $Y = A_1 \cup A_2 \cup \cdots \cup A_k$ such that 
\[\conv {A_1} \cap \conv {A_2} \cap \cdots \cap \conv {A_k} \neq \emptyset.\]
The {\em $k$th Tverberg number} $t_k(X, \cal C)$ is defined as the minimal integer $n$ (if it exists) such that every subset $Y\subset X$ with $|Y|\geq n$ admits a Tverberg $k$-partition. Note that $t_2(X, \cal C) = r(X, \cal C)$. 
\end{definition}

It is not hard to show that if $r(X, \cal C)$ is bounded, then so is $t_k(X, \cal C)$ for every $k\geq 3$. See e.g.  \cite[Exercise 8.3:1]{mato}.
A long-standing open problem raised by Calder and independently by Eckhoff (also known in the litterature the ``partition conjecture'' or ``Eckhoff's conjecture'') claimed that the inequality
\[
t_k(X, \cal C) \leq [r(X, \cal C)-1](k-1)+1    
\]
should hold for every convexity space with a bounded Radon number. This would be a far-reaching generalization of Tverberg's theorem \cite{tverberg} which asserts that the expression above holds with equality 
for the standard convexity on $\mathbb{R}^d$. 

\smallskip

In 2010 Bukh \cite{bukh} disproved the  partition conjecture by showing the following.

\begin{theorem}
    For every $k \geq 3$ there exists a convexity space $(X, \cal C)$ with $r(X, \cal C) = 4$ and $t_k(X, \cal C)\geq 3k-1$.
\end{theorem}
By combining Bukh's counterexample with the disjoint sum construction (see Definition 1.14, page 17 and Theorem 3.4, page 207 of \cite{vandevel}) provides, for every $r\geq 4$ and $k\geq 3$, a convexity space $(X, \cal C)$ where $t_k(X, \cal C) \geq (k-1) (r-1) + 1 + \lfloor \frac{r-1}{3} \rfloor$.

In the same paper, Bukh also proved an upper bound for the $k$th Tverberg number. He showed that $t_k(X, \cal C)  = O(k^2\log^2k)$, where the big $O$ is hiding a constant depending on the Radon number. He also suggested that the true upper bound might even be linear in $k$. 
This was recently proved by P{\'a}lv{\"o}lgyi.

\begin{theorem}[P{\'a}lv{\"o}lgyi \cite{domotor}] \label{t:linear-tverberg}
For any convexity space $(X, \cal C)$ with a finite Radon number, we have $t_k(X, \cal C) \leq c \cdot k$, where $c = c(r)$ is a constant depending only on $r = r(X, \cal C)$.
\end{theorem}

The proof of Theorem \ref{t:linear-tverberg} uses methods developed by Bukh \cite{bukh} together with a {\em fractional Helly theorem} for convexity spaces (which will be discussed below). A similar approach was also used in \cite[Proposition 10]{akmm} where they establish a ``weak Tverberg theorem'' in convexity spaces of the form $(X, F^\cap)$.

\smallskip

The main drawback of this approach is that it requires a somewhat enormous  constant $c(r)$, roughly $r^{r^{r^{\log r}}}$, which is most likely very far from the truth. The following question appears in \cite{bara-sob, domotor}.

\begin{problem}
Does there exist an absolute constant $c$ such that 
\[t_k(X, \cal C) \leq c  \cdot r(X, \cal C) \cdot k\] holds for every convexity space $(X, \cal C)$ with a bounded Radon number?
\end{problem}

\medskip

\paragraph{\bf Helly type invariants and properties}
We now turn our attention to Helly type properties. 
As a matter of fact, these are properties  of arbitrary set systems $(X, \cal S)$ and are not necessarily restricted  to convexity spaces. In particular, the convex hull operator is not used in any of their definitions. 

\begin{definition}
Let $(X, \cal S)$ be a set system.
The {\em Helly number} $h(X, \cal S)$ is  the maximal integer $n$ (if it exists) for which there exists a family $F\subset \cal S$ of size $n$, where $\cap F = \emptyset$ and 
$F\setminus \{K\}$ is intersecting for all $K\in F$.  In other words, the Helly number is the maxmial size of a minimally non-intersecting family. 
\end{definition}

Helly's theorem \cite{helly} states that the Helly number for the standard convexity on $\mathbb{R}^d$ equals $d+1$, while it is a simple exercise to show that the Helly number for the box convexity on $\mathbb{R}^d$ equals 2. 
For the lattice convexity on $\mathbb{Z}^d$, the Helly number equals $2^d$ \cite{JPD}.  
In general, Levi \cite{levi} showed that if a convexity space has bounded Radon number, then its Helly number is also bounded, and in particular $h(X, \mathcal{C})< r(X, \cal C)$.

\smallskip

While the Helly number is a classical and important invariant of a convexity space, here we focus our attention on some variations inspired by the colorful \cite{col-hell} and fractional \cite{katch-liu} versions of Helly's theorem.

\begin{definition}
A set system $(X, \cal S)$ has the {\em colorful Helly property} if there exists an integer $m$ such that for any finite families $F_1, F_2, \dots, F_m \subset \cal S$ with $K_1\cap K_2\cap \cdots \cap K_m \neq \emptyset$ for all choices $K_1\in F_1, K_2\in F_2, \dots, K_m\in F_m$, there is one of the $F_i$ that is intersecting. 
The {\em colorful Helly number} $h_c(X, \cal S)$ is the smallest integer $m$ (if it exists) for which $(X, \cal S)$ has the colorful Helly property. 
\end{definition}

The colorful Helly theorem \cite{col-hell} states that the standard convexity on $\mathbb{R}^d$ has colorful Helly number $d+1$. The colorful Helly number for the box convexity on $\mathbb{R}^d$ also equals $d+1$. Using elementary collapsibility arguments \cite[Lemma 20]{akmm}, it can be shown that the lattice convexity on $\mathbb{Z}^d$ has colorful Helly number $2^d$. See \cite[Theorem 2.1]{top-col-hel} for the relation between collapsibility and the colorful Helly number.
By taking $F_1 = F_2 =\cdots = F_m$, we observe that
the Helly number is always bounded above by the colorful Helly number, that is, for any set system satisfying the colorful Helly property we have $h(X, \cal S) \leq h_c(X, \cal S)$.

\smallskip

The colorful Helly property has not been investigated in general convexity spaces until quite recently, when it was shown that it is closely related to the Radon number. 

\begin{theorem}[Holmsen and Lee \cite{holmsen-arx}] \label{t:radon-to-colorful}
If a convexity space $(X, \cal C)$ has a bounded Radon number, then it has the colorful Helly property. In particular,  $h_c(X, \cal C)\leq m$, where $m = m(r)$ is a constant depending only on $r = r(X, \cal C)$.
\end{theorem}

The proof of Theorem \ref{t:radon-to-colorful} uses Tverberg partitions, and relies on the fact that the Tverberg numbers are bounded by a function of the Radon number. The bound on the colorful Helly number obtained in the proof of Theorem \ref{t:radon-to-colorful} can be expressed in terms of Stirling numbers of the second kind (the number of unordered partitions of $[n]$ into $k$ nonempty parts), and gives an upper bound on $m(r)$ which is roughly $r^{r^{\log r}}$. On the other hand, the case of box convexity shows that $m(r)$ grows at least exponentially in $r$. We have little reason to believe that either  of these bounds are optimal, but we lack methods for constructing interesting convexity spaces with bounded Radon number. One of the challenges is that the Radon number is not monotone with respect to deleting points from $X$ (i.e. restricting $\cal C$ to a subset of $X$).

\begin{problem}
Improve the bounds (upper and lower) on the colorful Helly number for a convexity space with Radon number $r$.
\end{problem}

\begin{definition}
A set system
$(X, \cal S)$ has the {\em fractional Helly property} if there exists an integer $k$ and a function $\beta:(0,1) \to (0,1)$ such that every finite family $F\subset \cal S$ with at least $\alpha\binom{|F|}{k}$ intersecting $k$-tuples, contains an intersecting subfamily of size at least $\beta(\alpha)|F|$. In this case we  also say that $(X, \cal S)$ satisfies the fractional Helly property for $k$-tuples. 
The {\em fractional Helly number} $h_f(X, \cal S)$ is the smallest integer $k$ for which $(X, \cal S)$ satisfies the fractional Helly property.
\end{definition}

The fractional Helly theorem \cite{katch-liu} asserts that the fractional Helly number for the standard convexity on $\mathbb{R}^d$ equals $d+1$. The same fractional Helly number also holds for the box convexity on $\mathbb{R}^d$.
A remarkable theorem due to B{\'a}r{\'a}ny and Matou{\v s}ek \cite{lattice} asserts that the fractional Helly number for the lattice convexity on $\mathbb{Z}^d$ equals $d+1$. Another result of Matou{\v s}ek \cite{mato-vc} asserts that if the set system $\cal S$ has bounded VC-dimension, then $(X, \cal S)$ has the fractional Helly property. More spcifically, if the {\em dual shatter function}  of $\cal S$ satisfies $\pi^*_{\cal S}(m) = o(m^k)$, then $(X, \cal S)$ has fractional Helly number at most $k$. 

\smallskip

Even though the colorful and fractional versions of Helly's theorem are classical results by now (at least in the standard convexity on $\mathbb{R}^d$), their relationship was only recently discovered.

\begin{theorem} \label{t:col-to-frac} Let $(X, \cal S)$ be a set system with $h_c(X, \cal S) = m$.
Then $(X, \cal S)$ satisfies the fractional Helly property for $m$-tuples with a function $\beta(\alpha) = \Omega(\alpha^{m^{m}})$.  In particular, $h_f(X, \cal S) \leq h_c(X, \cal S)$.
\end{theorem}

This theorem is a consequence of a result on the clique number of uniform hypergraphs \cite{holmsen-cliques}. If the set system $(X, \cal S)$ has colorful Helly number $m$, and $F\subset \cal S$ is a finite subfamily, we consider the $m$-uniform hypergraph $H$ whose vertices correspond to the members of $F$ and edges correspond to intersecting $m$-tuples in $F$. (In other words, $H$ is the set of $(m-1)$-faces of the nerve complex of $F$.) The colorful Helly property can be interpreted as forbidding certain subhypergraphs in $H$, and the main result of \cite{holmsen-cliques} implies that if $H$ has many edges, then $H$ must contain a large clique. Since $h(X, \cal S) \leq h_c(X, \cal S)$, it follows that a clique in $H$ corresponds to an intersecting subfamily in $F$. 

\smallskip

The standard convexity on $\mathbb{R}^d$ satisfies the fractional Helly property for $(d+1)$-tuples with the function $\beta(\alpha) = 1-(1-\alpha)^{1/(d+1)}$. This was proved by Kalai \cite{kalai-upper} and independently by Eckhoff \cite{eckhoff}, and simple examples show that this is best possible. %The main result of \cite{holmsen-cliques} implies that if $m = h_c(X, \cal S)$, then $(X, \cal S)$ satisfies the fractional Helly property for $m$-tuples with a function $\beta(\alpha) = \Omega(\alpha^{m^{m-1}})$.

\begin{problem}
    Improve the bounds (upper and lower) on the function $\beta(\alpha)$ appearing in the fractional Helly theorem for set systems with $h_c(X, \cal S) = m$.
\end{problem}

Let $F$ be a finite family of subsets of a ground set $X$. 
The {\em transversal number} $\tau(F)$ is the smallest integer $k$ such that $F$ can be partitioned $F = F_1 \cup F_2 \cup \cdots \cup F_k$ such that each $F_i$ is intersecting. 
The {\em fractional transversal number} $\tau^*(F)$ is the minimum of $\sum_{x\in X}f(x)$ over all functions $f:X \to [0,1]$ such that $\sum_{x\in S}f(x)\geq 1$ for every $S\in F$.

\begin{definition}
A set system $(X, \cal S)$ has the {\em weak $\varepsilon$-net property} if there exists a function $g : \mathbb{Q} \to \mathbb{N}$ such that for any finite family $F\subset \cal S$ we have $\tau(F) \leq g(\tau^*(F))$.
\end{definition}

It is an important result that the standard convexity on $\mathbb{R}^d$ has the weak $\varepsilon$-net property \cite{weak-nets}, and determining the optimal asymptotic behavior of the function $g$ is regarded as a major open problem in discrete geometry. (For the current best bounds see \cite{lower-nets, natan}.) Bukh \cite{bukh} showed that every convexity space with $r(X,\cal C)\leq 3$ has the weak $\varepsilon$-net property, and asked whether this is true for any  convexity space with bounded Radon number. It is a central question in combinatorics to identify properties of set systems that imply a bound on the transversal number and in particular the weak $\varepsilon$-net property. See e.g. \cite[Problem 7]{akmm}. One of the main results of \cite{akmm} gives such a criterion. 

\begin{theorem}[Alon, Kalai, Matou{\v s}ek, Meshulam \cite{akmm}] \label{t:frac-to-eps}
Let $(X, \cal C)$ be a convexity space which satisfies the fractional Helly property.
Then $(X, \cal C)$ satisfies the weak $\varepsilon$-net property. In particluar, the weak $\varepsilon$-net property holds with a function $g(x) = a x^b$ where $a$ and $b$ are constants that depend only on $h_f(X, \cal C)$ and the fractional Helly function $\beta(\alpha)$.
\end{theorem}

The proof of Theorem \ref{t:frac-to-eps} relies on using the fractional Helly property to establish a ``weak Tverberg theorem'' \cite[Proposition 10]{akmm}. This in turn can be used in combination with the fractional Helly property to establish a selection lemma \cite[Proposition 11]{akmm}, which is then used to construct a weak $\varepsilon$-net. This proof strategy is modeled after one of the well-known proofs of the weak $\varepsilon$-net theorem for the standard convexity on $\mathbb{R}^d$. See e.g. \cite[Theorem 10.4.2]{mato}.

\smallskip

The weak $\varepsilon$-net property was recently investigated by Moran and Yehudayoff \cite{moran} from a  different point of view. A set system $(X, \cal S)$ is {\em compact} if every subsystem $\cal F \subset \cal S$ with $\cap \cal F = \emptyset$ contains a finite family $F \subset \cal F$ such that $\cap F = \emptyset$. 

\begin{theorem}[Moran and Yehudayoff \cite{moran}]\label{t:VC-to-eps} Let $(X, \cal S)$ be a compact set system with bounded VC-dimension and Helly number. Then the set system $(X, \cal S^\cap)$ generated by taking arbitrary intersections of members in $\cal S$ satisfies the weak $\varepsilon$-net property. In particular, the weak $\varepsilon$-net property holds with a function $g(x) = cx^{d\ln x}$ where $c$ and $d$ are constants depending only on the VC-dimension and the Helly number. 
\end{theorem}

The proof of this result differs significantly from other proofs for weak $\varepsilon$-nets, and relies on Haussler's packing lemma \cite{haussler}. 

\smallskip

In the same paper, Moran and Yehudayoff apply Theorem \ref{t:VC-to-eps} to a special class of convexity spaces which satisfy an additional separation axiom. In a convexity space, a {\em halfspace}  is a convex set $H$ whose complement $X\setminus H$ is also convex. A convexity space $(X, \cal C)$ is {\em separable} if for every set $K\in \cal C$ and point $p\in X\setminus K$ there exists a halfspace $H$ such that $K\subset H$ and $p\notin H$. See \cite[section 3.4]{vandevel} for this and other separation axioms in convexity spaces. 

\smallskip

It is a known that a separable convexity space is generated by its system of halfspaces \cite[Corollary 3.9]{vandevel}. That is, if $\cal B \subset \cal C$ is the system of halfspaces of a separable convexity space $(X, \cal C)$, then $(X, \cal C) = (X, \cal B^\cap)$.
It is not hard to show that the set system $\cal B$ of halfspaces in a separable convexity space has VC-dimension strictly less than $r(X, \cal C)$ \cite[Claim A.2]{moran}. Applying Theorem \ref{t:VC-to-eps} to the system of halfspaces, and keeping in mind that $h(X, \cal C) < r(X, \cal C)$, implies that any separable convexity space with bounded Radon number satisfies the weak $\varepsilon$-net property \cite[Theorem 1.2(1)]{moran}.

\smallskip

Another result from their paper \cite[Theorem 1.2(2) and Lemma 1.13]{moran} draws a further connection between the weak $\varepsilon$-net property and the Radon number in general convexity spaces. 
By a clever reduction to the chromatic number of Kneser graphs, they obtain the following.

\begin{theorem}[Moran, Yehudayoff \cite{moran}] \label{t:eps-to-radon}
    Let $(X, \cal C)$ be a convexity space. If $(X, \cal C)$ satisfies the weak $\varepsilon$-net property, then $r(X, \cal C)$ is bounded. In particular, if $r(X, \cal C) > r$, then there is a finite family $F\subset \cal C$ such that $\tau^*(F) = 4$ and $\tau(F) > r/2$. 
\end{theorem}

The discussion would not be complete without clarifying how the invariants above are related to Alon and Kleitman's $(p,q)$-theorem.

\begin{definition}
Let $p\geq q$ be integers. A family of sets  has the {\em $(p,q)$-property} if among any $p$ members of the family there are some $q$ of them that intersect. We say that a
 set system $(X, \cal S)$ {\em admits a $(p,q)$-theorem} if there exists an integer $N = N(p,q)$ such that every finite family $F\subset \cal S$ which satisfies the $(p,q)$-property has transversal number $\tau(F)\leq N$.
\end{definition}

Alon and Kleitman \cite{pq-alon} showed that the standard convexity on $\mathbb{R}^d$ 
admits a $(p,q)$-theorem for all $p\geq q \geq d+1$. B{\'a}r{\'a}ny and Matou{\v s}ek \cite{lattice} showed that this is also true for the lattice convexity on $\mathbb{Z}^d$, that is, the lattice convexity on $\mathbb{Z}^d$ admits a $(p,q)$-theorem for all $p\geq q \geq d+1$. It is important to note that their bound on the transversal number is  different from the one for the standard convexity on $\mathbb{R}^d$. See also \cite[Section 6]{eckhoff-pq} for a discussion of $(p,q)$-theorems for abstract nerve complexes.

\smallskip

More generally, the proof method of Alon and Kleitman reveals the following (see e.g. \cite[Theorem 8]{akmm}).

\begin{theorem} \label{t:gen-pq}
Suppose the set system $(X, \cal S)$ satisfies the fractional Helly property for $k$-tuples, and let $p\geq q \geq k$. 
Then any finite family $F \subset \cal S$ with the $(p,q)$-property satisfies $\tau^*(F)\leq T$, where $T$ is a constant that depends only on $p$, $q$, and the function $\beta(\alpha)$ appearing in the fractional Helly property.  
\end{theorem}

It follows that any convexity space $(X, \cal C)$ with the fractional Helly property for $k$-tuples, also admits a $(p,q)$-theorem for all $p\geq q \geq k$. This is one of the main result of \cite[Theorems 8 and 9]{akmm}. By Theorem \ref{t:gen-pq} we get an absolute bound on the fractional transversal number, which in turn implies an absolute bound on the transversal number by Theorem \ref{t:frac-to-eps}.

\smallskip

It is worth pointing out that a $(p,q)$ theorem does not imply a universal bound on the fractional Helly number. For instance, K{\'a}rolyi \cite{karoly} proved that for the box convexity on $\mathbb{R}^d$ there is a $(p,2)$ theorem for every $p\geq 2$ and $d\geq 1$, but in this case the fractional Helly number equals $d+1$.

\subsection{Summary}
Here is a brief summary of the results discussed in this section.

\begin{theorem}\label{t:equivalences}
The following are equivalent for any convexity space $(X, \cal C)$.
\begin{enumerate}
    \item $(X, \cal C)$ has a bounded Radon number.
\smallskip
    \item $(X, \cal C)$ has the colorful Helly property.
\smallskip
    \item $(X, \cal C)$ has the fractional Helly property.
\smallskip
\item $(X, \cal C)$ has the weak $\varepsilon$-net property.
\end{enumerate}
\end{theorem}

\begin{proof}
(1) $\implies$ (2)  is Theorem \ref{t:radon-to-colorful}. (2) $\implies$ (3) is Theorem \ref{t:col-to-frac}. (3) $\implies$ (4) is Theorem \ref{t:frac-to-eps}. (4) $\implies$ (1) is Theorem \ref{t:eps-to-radon}.
\end{proof}

\begin{corollary}\label{c:(p,q)}
Let $(X, \cal C)$ be a convexity space with Radon number  $r(X, \cal C)\leq r$. Then $(X, \cal C)$ admits a $(p,q)$ theorem for all $p\geq q \geq h_f(X, \cal C)$, where the bound on the transversal number depends only on $p$, $q$, and $r$. 
\end{corollary}

\begin{proof}
If the Radon number is bounded, then Theorem \ref{t:equivalences} guarantees that all the conditions of Theorem \ref{t:gen-pq} are satisfied. Moreover, by Theorems \ref{t:radon-to-colorful} and \ref{t:col-to-frac}, $(X, \cal C)$ satisfies the fractional Helly property with a function $\beta(\alpha)$ that  depends only  on the Radon number. %(See the discussions following Theorems \ref{t:col-to-frac} and \ref{t:gen-pq}.)
\end{proof}

\medskip

The following table gives the values of the main parameters we have discussed for some specific convexity spaces. Note that the Radon number for the lattice convexity on $\mathbb{R}^d$ is not known, so we just state the upper and lower bounds. 

\bigskip

\renewcommand{\arraystretch}{1.7}
\begin{tabular}{|c|c|c|c|c|}
\hline
$(X, \cal C)$ & $r(X, \cal C)$ & $h(X, \cal C)$ & $h_c(X, \cal C)$ & $h_f(X, \cal C)$ \\ \hline 
standard convexity on $\mathbb{R}^d$ & $d+2$ & $d+1$ & $d+1$ & $d+1$ \\ \hline
box convexity on $\mathbb{R}^d$ & $\Theta(\log d)$ & $2$ & $d+1$ & $d+1$ \\ \hline
lattice convexity on $\mathbb{Z}^d$ & $\Omega(2^d), O(d2^d)$ & $2^d$ & $2^d$ & $d+1$ \\ \hline
\end{tabular}

\bigskip

\noindent
For a general convexity space with Radon number $r = r(X, \cal C)$ the following bounds hold.
\[\begin{array}{ccc}
    h(X, \cal C) & \leq  & r-1, \\
    h(X, \cal C) & \leq  & h_c(X, \cal C), \\
    h_f(X, \cal C) & \leq & h_c(X, \cal C), \\    
    h_c(X, \cal C) & \leq & r^{r^{\log r}}, \\ 
    t_k(X, \cal C) & \leq  & r^{r^{r^{\log r}}} \cdot k.

\end{array}\]

\section{Convexity spaces arising from uniform hypergraphs} 

In this section we apply the theory of convexity spaces to colorings of uniform hypergraphs. Let $H = (V, E)$ be a finite $k$-uniform hypergraph with vertex set $V$ and edge set $E \subset \binom{V}{k}$. In order to avoid certain degenerate situations, we assume throughout this section that $H$ is {\em nonempty}, meaning that $H$ contains at least one edge. 

\smallskip

The following standard terminology for hypergraphs will be used throughout this section. An {\em independent set} in $H$ is a subset of vertices which does not contain any edge of $H$. A {\em proper coloring} of $H$ is a partition of the vertex set $V = V_1 \cup V_2 \cup \cdots \cup V_k$ where each $V_i$ is an independent sets. The {\em chromatic number} $\chi(H)$ is the smallest number of parts in a proper coloring of $H$.  
A {\em clique} in $H$ is a subset of vertices in which every $k$-tuple is an edge of $H$, and the {\em clique number} $\omega(H)$ is the maximum number of vertices in a clique in $H$. A class of hypergraphs $\cal H$ is {\em $\chi$-bounded} if there exists a function $f: \mathbb{N} \to \mathbb{N}$ such that $\chi(H) \leq f(\omega(H))$ for every $H\in {\cal H}$.

\subsection{The associated convexity space}
Here we give a canonical construction of a convexity space derived from a uniform hypergraph $H = (V, E)$. For graphs, this construction appeared implicitly in \cite[Section 7.2]{akmm} and was also discussed in \cite{lssx}.

Let $X$ be the family of {\em maximal} independent sets in $H$, and for a subset $W \subset V$ define the set 
\[\es W = \{\sigma \in X : W\subset \sigma\}.\]
Observe that $\es\emptyset = X$, $\es V = \emptyset$ (since $H$ is nonempty), and $\es A \cap \es B = \es {A\cup B}$. Next, define the set system ${\cal C} \subset 2^{X}$ by setting
\[ {\cal C} = \{ \es W : W\subset V \}.\]
If follows from the observations above that $(X, {\cal C})$ is a convexity space, which we call the {\em associated convexity space} of $H$ and denote by $(X, \cal C)_H$. Note that this is a particular instance of the convexity space given in Example \ref{ex:spaces}\eqref{ex:dual}.

\smallskip

The convex hull operator of the associated convexity space has a particularly simple characterization.

\begin{lemma}\label{l:convexhulls}
    Let $H$ be a uniform hypergraph and $(X,  {\cal C})_H$ its associated convexity space. For a subset $A \subset X$ we have $\conv A = \es {\cap A}$.
\end{lemma}

\begin{proof}
It follows from the definition of the associated convexity space that  $\sigma \in \es W$ if and only if $W\subset \sigma$. Therefore $A \subset \es W$ if and only if $W \subset \cap A$. The claim now follows from the fact that $U\subset W$ implies $\es W \subset \es U$.
\end{proof}

Certain key properties of a hypergraph and its associated convexity space are related by the following.

\begin{lemma}\label{l:dictionary}
Let $H = (V, E)$ be a $k$-uniform hypergraph with associated convexity space $(X, {\cal C})_H$. For the family  $F = \{\es {\{v\}}  : v\in V\} \subset {\cal C}$ we have $\tau(F) = \chi(H)$.
\end{lemma}

\begin{proof}
Observe that a transversal for $F$ is equivalent to a collection of maximal independent sets $\sigma_1, \sigma_2, \dots, \sigma_m \in X$ such that $V\subset \sigma_1\cup \sigma_2 \cup \cdots \cup \sigma_m$. Since any partition of $V$ into independent sets can be extended to a covering of $V$ by maximal independent sets  it follows that $\tau(F) = \chi(H)$.
\end{proof}

\medskip

\subsection{The hypergraph class 
\texorpdfstring{${\cal H}_k(m)$}{}}
The construction of the associated convexity space does not actually use the uniformity of the hypergraph, and could be extended to any hypergraph. However, it is the specific application we now describe that will require the hypergraphs to be uniform.

\begin{definition} \label{d:tkm} For integers $m\geq k \geq 2$, let $H= (V, E)$ be a $k$-uniform hypergraph with at least $m$ edges. We say that $H$  has {\em property $T_k(m)$} if among every $m$ (not necessarily distinct) edges there exists $k$ of them that have a system of distinct representatives (SDR) which is an edge in $H$. In other words, for any $m$ edges $e_1, e_2, \dots, e_m \subset E$ there exists
indices $1\leq i_1 < i_2   < \cdots < i_k\leq m$ and
distinct vertices $v_1\in e_{i_1}, v_2 \in e_{i_2}, \dots, v_k\in e_{i_k}$ such that the $k$-tuple $\{v_1,v_2,\dots, v_k\} \in E$.

\smallskip

For integers $m\geq k \geq 2$ let $\cal{H}_k(m)$ denote the class of all $k$-uniform hypergraphs that satisfy property $T_k(m)$. 
\end{definition}

It is instructive to understand the graph class ${\cal H}_2(m)$. Suppose $G\in \cal{H}_2(m)$ and consider $m$ edges $e_1, e_2, \dots, e_m$ of $G$. If two of these edges share a vertex, say $e_{i_1} = \{u,v\}$ and $e_{i_2} = \{u,w\}$, then the SDR condition of  property $T_2(m)$ is automatically satisfied by setting $v_1 = u$ and $v_2 = w$. Thus property $T_2(m)$ only imposes a restriction when the edges $e_1, e_2, \dots, e_m$ are pairwise disjoint, that is, when they form a matching. Consequently,  the graph class $\cal{H}_2(m)$ consists of all graphs with no {\em induced matching} of size $m$. 

\smallskip

The main result of this section is the following.

\begin{theorem} \label{t:chi-bounded}
For every $m\geq k \geq 2$, the hypergraph class ${\cal H}_k(m)$ is $\chi$-bounded. 
\end{theorem}

For $k=2$, Theorem \ref{t:chi-bounded} asserts that the class of graphs with no induced matching of size $m$ is $\chi$-bounded. 
This is a special case of a well-known result of Gy{\'a}rf{\'a}s \cite[Theorem 2.1]{gyarfas} which states that the class $\cal{P}_\ell$ consisting of all graphs with no induced path of length $\ell$ is $\chi$-bounded. Since a path of length $\ell = 3m-1$ contains a matching of size $m$ it follows that $\cal{H}_2(m) \subset \cal{P}_{3m-1}$, and so the case $k=2$ of Theorem \ref{t:chi-bounded} is implied by Gy{\'a}rf{\'a}s' theorem. 

\smallskip

For $k>2$, the only proof of Theorem \ref{t:chi-bounded} that we are aware of uses the theory of convexity spaces. The crucial observation is the following. (We are grateful to Attila Jung for this observation.)

\begin{prop} \label{p:color_bound}
Let $H$ be a $k$-uniform hypergraph. Then $H \in \cal{H}_k(m)$ if and only if the colorful Helly number of $(X, {\cal C})_H$ is at most $m$.
\end{prop}

\begin{proof}
Fix a $k$-uniform hypergraph $H = (V,E)$. We note that every convex set in 
$(X, {\cal C})_H$ is the intersection of some members of $F = \{\es{v} \: : \: v\in V \}$, and it follows that $h(X, {\cal C})_H = h(X, F)$ and $h_c(X, {\cal C})_H = h_c(X, F)$.  It therefore suffices to show that $H\in \cal{H}_k(m)$ if and only if $h_c(X,F) \leq m$. 

The colorful Helly number of $(X,F)$ is at most $m$ if and only if for any finite families $F_1, F_2, \dots, F_m \subset F$, where the intersection of the members of $F_i$ is empty,
there exists convex sets $K_{1} \in F_{1}, K_{2}\in F_{2}, \dots, K_{m}\in F_{m}$ such that $K_{1} \cap K_{2} \cap \cdots \cap K_{m} = \emptyset$.
Since the Helly number of $(X,F)$ equals $k$ we may assume that $|F_i| = k$, for every $i$, and consequently we get $F_i  = \{\es{v_1}, \es{v_2}, \dots, \es{v_k} \}$ where $e_i = \{v_1, v_2, \dots, v_k\}$ is an edge of $H$. There exists sets $K_i\in F_i$ such that $K_1 \cap K_2 \cap \cdots \cap K_m = \emptyset$ if and only if some $k$ of the $K_i$ have empty intersection (by the Helly number), and this in turn is  equivalent to the existence of an SDR for the edges $e_1, e_2, \dots, e_m$ which is an edge in $H$.  
\end{proof}

\begin{proof}[Proof of Theorem \ref{t:chi-bounded}]
Let $H$ be a $k$-uniform hypergraph in ${\cal H}_k(m)$. By Theorem \ref{t:col-to-frac} and Proposition \ref{p:color_bound} it follows that $h_f(X, {\cal C})\leq h_c(X, {\cal C}) \leq  m$, and by applying Theorem \ref{t:equivalences} and Corollary \ref{c:(p,q)}, it follows that $(X, {\cal C})_H$ admits a $(p,q)$ theorem for all $p\geq q \geq m$, where the bound on the transversal number depends only on $p$, $q$, and $m$.

Let $R_k(a, b)$ denote the hypergraph Ramsey number which guarantees that for any $k$-uniform hypergraph on $R$ vertices there exists an independent set of size $a$ or a clique of size $b$. 
For our given $H\in \cal{H}_k(m)$, let $p = R_k(m, \omega(H)+1)$. If $H$ has less than $p$ vertices, then $\chi(H) <  p/(k-1)$. Otherwise, every subset of $p$ vertices contains an independent set of size $m$, which means that the family $F = \{\es v : v\in V\} \subset \cal{C}$ has the $(p,m)$-property, and by Lemma \ref{l:dictionary} and Corollary \ref{c:(p,q)}  we have \[\chi(H) = \tau(F) \leq N, \] where $N$ is a constant depending only on $m$, $k$,  and $\omega(H)$.
\end{proof}

As a consequence of Theorem \ref{t:equivalences} and Proposition \ref{p:color_bound} it follows that Radon number of the associated convexity space $(X, {\cal C})_H$ is uniformly bounded in terms of $k$ and $m$, for every $H\in {\cal H}_k(m)$. Here we give a direct argument for bounding the Radon number which might be of independent interest. (This direct argument also leads to a better bound than the one obtained by following the equivalences of Theorem \ref{t:equivalences}.) 

\begin{prop}\label{p:radon_bound}
For every $H\in \cal{H}_k(m)$ we have  $r(X, {\cal C})_H \leq \binom{m}{k}\cdot 2^k$. 
\end{prop}

The proof of Proposition \ref{p:radon_bound} requires the following construction.

\begin{lemma}\label{l:set_pairs} Let $m, k, N$ be integers satifying $m\geq k \geq 2$ and $N \geq \binom{m}{k}\cdot 2^k$.
There exists a family of pairs of sets $\big\{(A_i,B_i) \big\}_{i=1}^m$ satisfying the following properties:
\begin{enumerate}[(i)]
\item  \label{pr:partition} $A_i \cup B_i$ is a partition of $[N]$ for every $1\leq i \leq m$. That is, $A_i$ and $B_i$ are both nonempty, disjoint, and their union equals $[N]$. 
\smallskip

\item \label{pr:transversal} For every $1\leq i_1 < i_2 < \cdots < i_k \leq n$ and every choice $W_1\in \{A_{i_1}, B_{i_1}\}$, $W_2\in \{A_{i_2}, B_{i_2}\}$, $\dots$, $W_k\in \{A_{i_k}, B_{i_k}\}$ we have
\[W_1\cap W_2 \cap \cdots \cap W_k \neq \emptyset.\]
\end{enumerate}
\end{lemma}

\begin{proof}
We first give a construction for $m=k$.   Let $T_1, T_2, \dots, T_{2^k}$ be the distinct subsets of $[k]$. For $1\leq i \leq 2^k$ and $1\leq j \leq k$ define the sets $A_j$ by the rule
\begin{equation}\label{eq:partition_rule}
i\in A_j \iff j\in T_i,    
\end{equation}
and set $B_j = [2^k]\setminus A_j$. Observe that $|A_i| = |B_i| = 2^{k-1}$ and that   $A_i\cup B_i$ is a partition of $[2^k]$ for all $1\leq i\leq 2^k$. It remains to verify property \eqref{pr:transversal}. For a given choice $W_1, W_2, \dots, W_k$, where $W_i \in \{A_i, B_i\}$, consider the set 
\[X = \{j : W_j = A_j\} \subset [k].\]
Then $X = T_i$ for some $i$ and by \eqref{eq:partition_rule} we have $i\in W_1 \cap W_2 \cap \cdots \cap W_k$.

\smallskip

Let us illustrate the construction with an example for $m=k=3$. If we set 
\[
\begin{array}{cccc}
T_1 = \emptyset, &
T_2 = \{1\},&
T_3 = \{1,2\},&
T_4 = \{1,3\}, \\
T_5 = \{1,2,3\}, &
T_6 = \{2\}, &
T_7 = \{2,3\}, &
T_8 = \{3\}, 
\end{array}
\]
then rule \eqref{eq:partition_rule} gives us the family of partitions
\[
\begin{array}{ccc}
A_1 = \{2, 3, 4, 5\} & , &B_1 = \{1, 6, 7, 8\}\\
A_2 = \{3, 5, 6, 7\} & , &B_2 = \{1, 2, 4, 8\}\\
A_3 = \{4, 5, 7, 8\} & , &B_3 = \{1, 2, 3, 6\}.
\end{array}\]

\smallskip

Here is a sketch for  general  $m>k$. First we note that the family of partitions of $[2^k]$ constructed above can be extended to a family of partitions of a {\em larger} set by simply adding additional elements to either $A_i$ or $B_i$ for every $1\leq i \leq k$. This works since the extended family of partitions retains property \eqref{pr:transversal}. So to construct a family of partitions $\big\{ (A_i, B_i)\big\}_{i=1}^m$ we can iterate the construction above over all  $k$-tuples of $[m]$. If a given $k$-tuple does not satisfy property \eqref{pr:transversal}, we add $2^k$ additional elements to the ground set and distribute them among the $(A_i,B_i)$ of the given $k$-tuple according to rule \eqref{eq:partition_rule}, after which we  distribute these additional elements arbitrarily to the other $(A_j, B_j)$ not in the $k$-tuple. Repeating this process at most $\binom{m}{k}$ times assures that each $k$-tuple of partitions satisfies property \eqref{pr:transversal}. Finally, if necessary, we can include additional dummy elements to obtain partitions of $[N]$, which completes the construction. 
\end{proof}

\begin{proof}[Proof of Lemma \ref{p:radon_bound}]
Set $N = \binom{m}{k} \cdot 2^k$. Let $H$ be a $k$-uniform hypergraph and assume the Radon number $r(X, {\cal C})_H  >  N$. Our goal is to show that $H\notin {\cal H}_k(m)$, that is, $H$ contains $m$ edges that violate Property $T_k(m)$. 

\medskip

Let $\big\{(A_i,B_i)\big\}_{i=1}^m$ be a family of partitions of $[N]$ which satisfies the properties of Lemma \ref{l:set_pairs}.
By the assumption that $r(X, {\cal C})_H > N$, 
there exists a subset $\{ \sigma_1, \sigma_2, \dots, \sigma_N \} \subset X$ 
which does not admit a Radon partition. 
For each $i\in [m]$ define the sets
\[Y_i = \textstyle \bigcap_{j\in A_i} \sigma_j \;\; \text{ and } \; \; Z_i = \bigcap_{j\in B_i} \sigma_j.\]
Observe that $Y_i$ and $Z_i$ are independent sets in $H$, so by Lemma \ref{l:convexhulls} we have 
\[\conv {\{\sigma_j\}_{j\in A_i}} = \es{Y_i} \;\; \text{ and } \;\; \conv {\{\sigma_j\}_{j\in B_i}} = \es{Z_i}, \]
and since $\{ \sigma_i \}_{i=1}^N$ has no Radon partition we have $S_{Y_i}\cap S_{Z_i} = \emptyset$, which means that $Y_i\cup Z_i$ is not independent in $H$.  Consequently, we obtain a sequence of edges $(e_1, e_2, \dots, e_m)$ in $H$  where $e_i\subset Y_i \cup Z_i$ for every $i\in [m]$.

%\begin{claim}\label{c:norepetition}
%An edge in $H$ appears at most $2^{k-1}$ times in the sequence $(e_1, e_2, \dots, e_n)$. 
%\end{claim}

%\noindent To justify this claim, suppose $e = e_{i_1} = e_{i_2} = \cdots = e_{i_r}$. For every $j\in [r]$ we define the sets 
%\[\alpha_j = e \cap Y_{i_j} \;\; \text{and} \;\; \beta_j = e \cap Z_{i_j},\] thereby
%obtaining a sequence  $(\alpha_1, \beta_1, \alpha_2, \beta_2, \dots, \alpha_r, \beta_r)$ of (nonempty) subsets of $e$. The claimed bound follows by showing that this sequence consists of pairwise distinct subsets of $e$, implying $2r \leq 2^k$. Clearly we have $\alpha_j\neq \beta_j$, so for contradiction suppose $\alpha_j = \alpha_\ell$ for some $j \neq \ell$. (This also covers the case $\alpha_j = \beta_\ell$ by interchanging $A_{i_\ell}$ and $B_{i_\ell}$.)

%\smallskip

%Since $\alpha_{j}\subset Y_{i_j}$ and $\alpha_{\ell}\subset Y_{i_\ell}$, we have $e \setminus \alpha_{\ell}\subset \beta_{\ell}\subset Z_{i_\ell}$, and therefore $e\subset \alpha_j \cup \beta_\ell$. 
%By property \eqref{pr:transversal} of Lemma \ref{l:set_pairs} there exists an integer $t\in A_{i_j}\cap B_{i_\ell} \subset [N]$, from which it follows that $\alpha_{j}\subset \sigma_t$ and $\beta_{\ell} \subset \sigma_t$. But then $e\subset \sigma_t$ which is impossible since $\sigma_t$ is a (maximal) independent set in $H$. This proves the claim.

\smallskip

%By Claim \ref{c:norepetition} it follows that the sequence $(e_1, e_2, \dots, e_n)$ contains a set of $m$ distinct edges $\{e_{i_1}, e_{i_2}, \dots, e_{i_m}\}$ since $n = 2^{k-1}m$. 
To complete the proof we show that these $m$ edges violate property $T_k(m)$. To this end, 
consider an arbitrary set of $k$ distinct vertices \[v_1\in e_{j_1}, v_2\in e_{j_2}, \dots, v_k\in e_{j_k},\] for some $k$-tuple of distinct indices $\{j_1, j_2, \dots, j_k\} \subset \{i_1, i_2, \dots, i_m\}$. Recall that $e_{j_\ell} \subset Y_{j_\ell}\cup Z_{j_\ell}$ for every $1\leq \ell\leq k$ and therefore \[v_1\in R_{j_1}, v_2 \in R_{j_2}, \dots, v_k \in R_{j_k}\] for some choice 
$R_{j_1}\in \{Y_{j_1},  Z_{j_1}\}, R_{j_2}\in \{Y_{j_2},  Z_{j_2}\}, \dots, 
R_{j_k}\in \{Y_{j_1},  Z_{j_k}\}$. Equivalently, for every $1\leq \ell \leq k$ there is a choice $W_{j_\ell} \in \{ A_{j_\ell}, B_{j_\ell}\}$ such that \[\textstyle v_\ell \in \bigcap_{t\in W_{j_\ell}} \sigma_t.\]
By Property \eqref{pr:transversal} of Lemma \ref{l:set_pairs}, there is an integer $t\in W_{j_1}\cap W_{j_2} \cap \cdots \cap W_{j_k}$, which means that $v_j\in \sigma_t$ for every $1\leq j \leq k$. Therefore $\{v_1, v_2, \dots, v_k\}$ is not an edge in $H$,  since $\sigma_t$ is a (maximal) independent set in $H$. This shows that $H$ does not satisfy Property $T_k(m)$ and completes the proof.
\end{proof}

\begin{remark}\label{r:hkm-bound}
While the upper bound on the Radon number given by Proposition \ref{p:radon_bound} is rather crude, it is not hard to see that it must grow as function of $m$ and $k$. For instance, if $H$ is the $k$-uniform hypergraph consisting of $m = 2^r-1$ pairwise disjoint edges (i.e. a matching of size $m$), then the Radon number $r(X, \cal C)_H \geq r+k$. (We leave the proof to the reader.)
\end{remark}

\section{Concluding remarks} 

\noindent $\diamond$ \hspace{1ex}
In section 2 we saw that the various Helly type invariants for a convexity space can be bounded in terms of its Radon number. It would be highly interesting to find the optimal dependencies. 

\medskip
\noindent $\diamond$ \hspace{1ex}
One aspect which we (purposely) omitted from the discussion in section 2 is the theory of {\em nerve complexes} and the {\em Leray number}. See for instance \cite{tancer} for a survey on nerve complexes of families of convex sets in $\mathbb{R}^d$. Let $F = \{A_1, \dots, A_n\}$ be a family of subsets of a ground set $X$. The nerve of $F$ is the simplicial complex 
\[N(F) = \{ \sigma \subset [n] : \textstyle \bigcap_{i\in \sigma} A_i \neq \emptyset\}.\]
Much of the theory of intersection patterns of convex sets in $\mathbb{R}^d$ has been studied via nerve complexes. An important parameter in this context is the Leray number. A simplicial complex is $d$-Leray if $\tilde{H}_i(L, \mathbb{Q}) = 0$ for all $i\geq d$ and all induced subcomplexes $L\subset K$. The Leray number $\ell(K)$ is the smallest integer $d$ for which $K$ is $d$-Leray. It is a basic fact due to Wegner that the nerve complex of a finite family of convex sets in $\mathbb{R}^d$ is $d$-Leray, but the class of $d$-Leray complexes is much more general. As a matter of fact,  much of the known Helly-type theorems for families of convex sets in $\mathbb{R}^d$ have been extended to $d$-Leray complexes, see e.g. \cite{akmm, kalai-upper, top-col-hel}. The following natural question arises.

\smallskip

\begin{question*}
Does there exist a function $d   : \mathbb{N} \to \mathbb{N}$ with the following property? For any finite family $F\subset \cal C$ of convex sets in a convexity space $(X, \cal C)$ with Radon number at most $r$, the nerve $N(F)$ has Leray number at most $d(r)$.
\end{question*}

\smallskip

An affirmative answer to this question would imply bounds on the fractional and colorful Helly numbers, as well as a $(p,q)$-theorem. It turns out, however, that the answer is {\em negative}. This follows from constructions of Januszkiewicz and {\'S}wi{\k a}tkowski \cite{janus} (see also Osajda \cite{osajda}), who showed that for every positive integer $n$ there exists a graph $G_n$, with no induced $C_4$-subgraph, such that the clique complex $K(G_n)$ has nonvanishing $n$-dimensional rational homology. To see how this relates to convexity spaces we consider the complement graph $\overline{G_n}$ which has no induced matching of size 2. By Propositions \ref{p:color_bound} and \ref{p:radon_bound}, the associated convexity space of $\overline{G_n}$ has colorful Helly number 2 and Radon number at most 4. Recall that for a subset of vertices $v_1, \dots, v_k$ in $\overline{G_n}$, we have $\es{v_1}\cap \es{v_2}\cap \cdots \cap \es{v_k} \neq \emptyset$ if and only if $\{v_1, v_2, \dots, v_k\}$ is an independent set in $\overline{G_n}$, or equivalently a clique in $G_n$. Therefore the nerve complex of $F = \{\es{v} : v\in V(\overline{G_n})\}$ is isomorphic to the clique complex of $G_n$, which implies that it is impossible to bound the Leray number as a function of the Radon number.

\medskip
\noindent $\diamond$ \hspace{1ex}
It is likely that the bounds on the Helly type invariants could be improved a lot in the case of separable convexity spaces (discussed in connection with Theorem \ref{t:VC-to-eps}). For instance, in \cite{zuz} it is shown that if a separable convexity spaces has a bounded Radon number then the fractional Helly number is bounded by the {\em dual VC-dimension} of the system of halfspaces. Can we improve the bound on the Tverberg numbers for separable convexity spaces? 

\medskip
\noindent $\diamond$ \hspace{1ex} Can we find sufficient conditions that guarantee that a set system $(X, \cal S)$ has the fractional Helly property? One such condition is that the Radon number $r(X, \cal S^\cap)$ is bounded, but what can be said when $r(X, \cal S^\cap)$ is unbounded?   
\smallskip

A prototypical example comes from geometric transversal theory. Let $(\mathbb{R}^d, \cal C)$ denote the standard convexity on $\mathbb{R}^d$, and for $K\in \cal C$ let $K^*$ denote the set of affine hyperplanes that intersect $K$. If we define 
\[\cal S = \{K^* : K\in \cal C\}, \] we obtain a set system $(\mathbb{G}^d, \cal S)$ where $\mathbb{G}^d$ denotes the ``affine Grassmannian'' whose points represent affine hyperplanes in $\mathbb{R}^d$. This is an example of the convexity spaces explored in \cite{goodman}. It turns out that the Helly number $h(\mathbb{G}^d, \cal S)$ is unbounded, and consequently the Radon number $r(\mathbb{G}^d, \cal S^\cap)$ is unbounded as well. However, it was shown in \cite{hyperplanes} that the set system $(\mathbb{G}^d, \cal S)$ has the fractional Helly property for $(d+1)$-tuples.

\medskip
\noindent $\diamond$ \hspace{1ex} 
Are there other interesting classes of uniform hypergraphs that are $\chi$-bounded? 
Here is one example. 
Recall that  a set of $m$ edges $\{e_1, e_2, \dots, e_m\}$ in $H$ is a {\em sunflower} if there is a set $S$ such that  $e_i\cap e_j = S$ for all $i \neq j$.  
We say that a $k$-uniform hypergraph $H$ has property $\delta_k(m)$ if any sunflower of size $m$ in $H$ contains $k$ edges that have a system of distinct representatives which is an edge in $H$. 
In other words, the SDR condition in property $T_k(m)$ is only invoked on $m$-tuples of edges that form sunflowers. For integers $m\geq k \geq 2$  let  $\Delta_k(m)$ denote the class of all $k$-uniform hypergraphs that satisfy property $\delta_k(m)$.
It is straight-forward to show, using the sunflower lemma \cite{sunflower}, that $\Delta_k(m) \subset \cal{H}_k(t)$ for some $t = t(k,m)$, and so  Theorem \ref{t:chi-bounded} therefore implies that the class $\Delta_k(m)$ is $\chi$-bounded for every $m\geq k \geq 2$.

\medskip
\noindent $\diamond$ \hspace{1ex} 
Here we suggest a further generalization of the hypergraph properties discussed earlier. We say that a $k$-uniform hypergraph $H$ has property  $D_k(m)$ if for any set of {\em pairwise disjoint} edges $\{e_1, e_2, \dots, e_m\}$ there exists an edge $e\subset e_1\cup e_2 \cup\cdots \cup e_m$ and such that $|e\cap e_i|\leq 1$ for every $i$. In other words, the SDR condition from the $T_k(m)$ property is only invoked on $m$-tuples of edges that are pairwise disjoint. For integers $m\geq k \geq 2$, let $\cal D_k(m)$ denote the class of hypergraphs that satisfy property $D_k(m)$. We conjecture that the class $\cal D_k(m)$ is $\chi$-bounded for all $m\geq k \geq 2$. Note that the class ${\cal D}_k(m)$ is not contained in ${\cal H}_k(t)$ for any $t$, since the class ${\cal D}_k(m)$ contains all sunflowers with nonempty kernels. (See the previous remark.) Therefore Proposition \ref{p:color_bound} implies that there is no uniform bound on the colorful Helly number (nor the Radon number) for the associated convexity spaces of hypergraphs in ${\cal D}_k(m)$.

\end{document}